\documentclass[a4paper,12pt]{article}

\usepackage[left=2cm,right=2cm, top=2cm,bottom=2cm,bindingoffset=0cm]{geometry}

\usepackage{verbatim}
\usepackage{amsmath}
\usepackage{amsthm}
\usepackage{amssymb}
\usepackage{delarray}
\usepackage{cite}
\usepackage{hyperref}
\usepackage{mathrsfs}

\newcommand{\al}{\alpha}

\newcommand{\ga}{\gamma}
\newcommand{\de}{\delta}
\newcommand{\la}{\lambda}
\newcommand{\om}{\omega}

\newcommand{\eps}{\varepsilon}
\newcommand{\vv}{\varphi}
\newcommand{\iy}{\infty}

\theoremstyle{plain}

\newtheorem{thm}{Theorem}
\newtheorem{lem}{Lemma}

\newtheorem{cor}{Corollary}

\theoremstyle{remark}

\DeclareMathOperator*{\Res}{Res}

 \allowdisplaybreaks

\begin{document}

\begin{center}
{\large\bf An inverse spectral problem for the matrix Sturm-Liouville operator on the half-line 
}
\\[0.2cm]
{\bf Natalia Bondarenko} \\[0.2cm]
\end{center}

\vspace{0.5cm}

{\bf Abstract.} The matrix Sturm-Liouville operator with an integrable potential on the half-line is considered.
We study the inverse spectral problem, which consists in recovering of this operator by the Weyl matrix.
The main result of the paper is the necessary and sufficient conditions on the Weyl matrix of the non-self-adjoint
matrix Sturm-Liouville operator. We also investigate the self-adjoint case and obtain the characterization of
the spectral data as a corollary of our general result.  

\medskip


\medskip


\vspace{1cm}

{\large \bf 1. Introduction and main results} \\

Inverse spectral problems consist in recovering differential operators from their spectral characteristics.
Such problems arise in many areas of science and engineering, i.e. quantum mechanics, geophysics, astrophysics, electronics.
The most complete results were obtained in the theory of inverse spectral problems for scalar Sturm-Liouville operators
$-y'' + q(x) y$
(see monographs \cite{Mar77, Lev84, PT87, FY01} and references therein). The greatest progress in the study of 
Sturm-Liouville operators on the {\it half-line} has been achieved by V.A.~Marchenko \cite{Mar77}.
He has studied the inverse problem for the non-self-adjoint locally integrable potential by the generalized spectral function,
using the method of transformation operator. We also mention, that V.A.~Marchenko has solved the inverse scattering problem on the half-line. 
Later V.A.~Yurko has shown, that the inverse problem by the generalized spectral function is equivalent 
to the problem by the generalized Weyl function \cite{FY01}. These problems are closely related to the inverse 
problem  for the wave equation $u_{tt} = u_{xx} - q(x) u$. When the potential is integrable on the half-line,
the generalized Weyl function turns into the ordinary Weyl function. V.A. Yurko has studied inverse problems for
the Sturm-Liouville operator with the potential from $L(0, \iy)$ by the Weyl function and,
in the self-adjoint case, by the spectral data. He developed a constructive algorithm for the solution of these problems
and obtained necessary and sufficient conditions for the corresponding spectral characteristics.
The details are presented in \cite{FY01}. In this paper, we generalize his results to the matrix case. 

The research on the inverse matrix Sturm-Liouville problems started in connection with their applications
in quantum mechanics \cite{AM60}. Matrix Sturm-Liouville equations can also be used to describe
propagation of seismic \cite{Beals95} and electromagnetic waves \cite{Boutet95}.
Another important application is the integration of matrix nonlinear evolution equations, such as matrix KdV and boomeron equations \cite{CD77}.
The theory of matrix Sturm-Liouville problems have been actively developed during the last twenty years.
Trace formulas, eigenvalue asymptotics and some other aspects of direct problems 
were studied in papers \cite{Yang12, Carlson99, Carlson00, Chern98, Dwyer94}. The works
\cite{Malamud05, Yur06, CK09, MT10, Bond14} contain results of the most resent investigations of inverse problems
for matrix Sturm-Liouville operators on a finite interval. 

For the matrix Sturm-Liouville operator on the half-line, Z.S.~Agranovich and V.A.~Marchenko \cite{AM60} 
have made an extensive reseach on the inverse {\it scattering} problem, using the transformation operator method \cite{Mar77, Lev84}.
G. Freiling and V.A.~Yurko \cite{FY07} have started the investigation of the inverse {\it spectral} problem
for the non-self-adjoint matrix Sturm-Liouville operator.
They have proved the uniqueness theorem and provided a constructive 
algorithm for the solution of the inverse problem by the so-called Weyl matrix (the generalization of the scalar Weyl function \cite{Mar77}, \cite{FY01}).
Their approach is based on the method of spectral mappings (see \cite{FY01, Yur02}),
whose main ingredient is the contour integration in the complex plane of the spectral parameter $\la$.
We mention that a related inverse problem for the matrix wave equation was investigated in \cite{ABI91}.

In this paper, we study the inverse problem for the matrix Sturm-Liouville operator on the half-line
by the Weyl matrix. We present the necessary and sufficient conditions for the solvability of the inverse problem
in the general non-self-adjoint situation. As a particular case, we consider the self-adjoint problem,
and get the necessary and sufficient conditions on the spectral data of the self-adjoint operator.
Our method is based on the approach of \cite{FY07}.

Proceed to the formulation of the problem. 
Consider the boundary value problem $L = L(Q(x), h)$ for the matrix Sturm-Liouville equation
\begin{equation} \label{eqv}
    \ell Y: = -Y''+ Q(x) Y = \la Y, \quad x > 0,                                           
\end{equation}
\begin{equation} \label{IC}
    U(Y) := Y'(0) - h Y(0) = 0.                                
\end{equation}

Here $Y(x) = [y_k(x)]_{k = \overline{1, m}}$ is a column vector,
$\la$ is the spectral parameter, $Q(x) = [Q_{jk}(x)]_{j, k = 1}^m$ is an $m \times m$ matrix function with entries from $L(0, \iy)$, 
and $h = [h_{jk}]_{j, k = 1}^m$, where $h_{jk}$ are complex numbers.  

Let $\la = \rho^2$, $\rho = \sigma + i \tau$, and let for definiteness $\tau := \mbox{Im} \rho \ge 0$.
Denote by $\Phi(x, \la) = [\Phi_{jk}(x, \la)]_{j, k = 1}^m$ the matrix solution of equation \eqref{eqv}, 
satisfying boundary conditions $U(\Phi) = I_m$ ($I_m$ is the $m \times m$ unit matrix), 
$\Phi(x, \la) = O(\exp(i \rho x))$, $x \to \iy$, $\rho \in \Omega := \{ \rho \colon \mbox{Im}\, \rho \ge 0, \: \rho \ne 0\}$.
Denote $M(\la) = \Phi(0, \la)$. We call the matrix functions $\Phi(x, \la)$ and $M(\la)$ 
the {\it Weyl solution} and the {\it Weyl matrix} of $L$, respectively.
Further we show, that the singularities of $\Phi(x, \la)$ and $M(\la)$ 
coincide with the spectrum of the problem $L$.
The Weyl functions and their generalizations often appear in applications and in pure
mathematical problems for various classes of differential operators.
In this paper, we use the Weyl matrix as the main spectral characteristic and study the following problem.

\medskip
{\bf Inverse Problem 1.} Given the Weyl matrix $M(\la)$, construct the potential $Q$ and the coefficient $h$.
\medskip

The paper is organized as follows. 
In {\it Section~2}, we present the most important properties of the Weyl matrix
and briefly describe the solution of Inverse Problem~1, given in \cite{FY07}.
By the method of spectral mappings, the nonlinear inverse problem is transformed to the linear
equation in a Banach space of continuous matrix functions. In {\it Section~3}, we use this solution
to obtain our main result, necessary and sufficient conditions for the solvability of Inverse Problem~1. 
In the general non-self-adjoint situation, 
one has to require the solvability of the main equation in the necessary and sufficient conditions.
Of course, it not always easy to check this requirement, but one can not avoid it even 
for the scalar Sturm-Liouville operator (examples are provided in \cite{FY01}).
Therefore we are particularly interested in the special cases, when the solvability of the main equation
can be easily checked. First of all, there is the self-adjoint case, studied in {\it Sections~4} and {\it 5}. 
We introduce the spectral data and get their characterization. We also consider finite perturbations
of the spectrum in {\it Section~6}. In this case, the main equation 
turns into a linear algebraic system, and one can easily verify its solvability.

\bigskip

{\large \bf 2. Preliminaries}\\

In this section, we provide the properties of the Weyl matrix
and the algorithm for the solution of Inverse Problem~1 by the method
of spectral mappings. We give the results without proofs, one can read \cite{AM60, FY07} for more details.

Start with the introduction of the notation.
We consider the space of complex column $m$-vectors $\mathbb{C}^m$ with the norm
$$
    \| Y \| = \max_{1 \le j \le m} |y_j|, \quad Y = [y_j]_{j = \overline{1, m}},
$$
the space of complex $m \times m$ matrices $\mathbb{C}^{m \times m}$ with the corresponding induced norm
$$
\| A \| = \max_{1 \le j \le m} \sum_{k = 1}^m |a_{jk} |, \quad A = [a_{jk}]_{j, k = \overline{1, m}}.
$$
The symbols $I_m$ and $0_m$ are used for the unit $m \times m$ matrix and the zero $m \times m$ matrix, respectively.
The symbol $\dagger$ denotes the conjugate transpose.

We use the notation $\mathcal{A}(\mathcal{I}; \mathbb{C}^{m \times m})$ for a class of the matrix functions
$F(x) = [f_{jk}(x)]_{k = \overline{1, m}}$ with entries $f_{jk}(x)$ belonging to the class $\mathcal{A}(\mathcal{I})$
of scalar functions. The symbol $\mathcal{I}$ stands for an interval or a segment.
For example, the potential $Q$ belongs to the class $L((0, \iy); \mathbb{C}^{m \times m})$. 

Denote by $\Pi$ the $\la$-plane with a cut $\la \ge 0$, and $\Pi = \overline{\Pi} \backslash \{ 0 \}$;
note that here $\Pi$ and $\Pi_1$ must be considered as subsets of the Riemann surface
of the square root function. Then, under the map $\rho \to \rho^2 = \la$,
$\Pi_1$ corresponds to the domain $\Omega = \{\rho\colon \mbox{Im} \rho \ge 0, \rho \ne 0 \}$.

Let us introduce the matrix {\it Jost solution} $e(x, \rho)$.
Equation~\eqref{eqv} has a unique matrix solution $e(x, \rho) = [e_{jk}(x, \rho)]_{j, k = 1}^m$,
$\rho \in \Omega$, $x \ge 0$, satisfying the integral equation
\begin{equation} \label{inteqe}
e(x, \rho) = \exp(i \rho x) I_m - \frac{1}{2 i \rho} \int_x^{\iy} (\exp(i \rho (x - t)) - \exp(i \rho (t - x)) Q(t) e(t, \rho) dt.
\end{equation}
The matrix function $e(x, \rho)$ has the following properties:

($i_1$) For $x \to \iy$, $\nu = 0, 1$, and each fixed $\de > 0$,
\begin{equation}  \label{asymptex}
e^{(\nu)}(x, \rho) = (i \rho)^{\nu} \exp(i \rho x) (I_m + o(1)),
\end{equation}
uniformly in $\Omega_{\de} := \{ \mbox{Im}\, \rho \ge 0, \, |\rho| \ge \de\}$.

($i_2$) For $\rho \to \iy$, $\rho \in \Omega$, $\nu = 0, 1$,
\begin{equation}
e^{(\nu)}(x, \rho) = (i \rho)^{\nu} \exp(i \rho x) \left( 1 + \frac{\om(x)}{i \rho} + o(\rho^{-1}) \right),
\quad \om(x) := -\frac{1}{2} \int_x^{\iy} Q(t)\,dt,
\end{equation}
uniformly for $x \ge 0$.

($i_3$) For each fixed $x \ge 0$ and $\nu = 0, 1$, the matrix functions $e^{(\nu)}(x, \rho)$ are analytic 
for $\mbox{Im} \rho > 0$ and continuous for $\rho \in \Omega$.

($i_4$) For $\rho \in \mathbb{R} \backslash \{ 0\}$ the columns of the matrix functions $e(x, \rho)$ and $e(x, -\rho)$
form a fundamental system of solutions for equation \eqref{eqv}.

The construction of the Jost solution in the matrix case was given in the Appendix of \cite{BF14} for even more general situation
of the matrix pencil. In principle,
the the proof is not significantly different from the similar proof in the scalar case (see \cite[Section 2]{FY01}).

Along with $L$ we consider the problem $L^* = L^*(Q(x), h)$ in the form
\begin{equation} \label{eqv*}
\ell^* Z := -Z'' + Z Q(x) = \la Z, \quad x > 0, 
\end{equation}
\begin{equation} \label{IC*}
U^*(Z) := Z'(0) - Z(0)h = 0.
\end{equation}
where $Z$ is a row vector. Denote $\langle Z, Y \rangle := Z'Y - Z Y'$.
If $Y(x, \la)$ and $Z(x, \la)$ satisfy equations \eqref{eqv} and \eqref{eqv*}, respectively, then
\begin{equation} \label{wron}
 	\frac{d}{dx} \langle Z(x, \la), Y(x, \la) \rangle = 0,
\end{equation}
so the expression $\langle Z(x, \la), Y(x, \la) \rangle$ does not depend on $x$.
                                                                            
One can easily construct the Jost solution $e^*(x, \rho)$ of equation \eqref{eqv*}, satisfying the integral equation
\begin{equation} \label{inteqe*}
e^*(x, \rho) = \exp(i \rho x) I_m - \frac{1}{2 i \rho} \int_x^{\iy} \left(\exp(i \rho (x - t)) - \exp(i \rho (t - x)) \right) e^*(t, \rho) Q(t) dt.
\end{equation}
and the same properties ($i_1$)-($i_4$), as $e(x, \rho)$.

If $\rho \in \mathbb{R} \backslash \{0\}$, then
\begin{equation} \label{wronee}
 	\langle e^*(x, -\rho), e(x, \rho) \rangle = - 2 i \rho I_m.
\end{equation}
Indeed, by virtue of $\eqref{wron}$, the expression $\langle e^*(x, -\rho), e(x, \rho) \rangle$
does not depend on $x$. So one can take a limit as $x \to \iy$ and use asymptotics \eqref{asymptex}, in order to 
derive \eqref{wronee}.

Denote $u(\rho) := U(e(x, \rho)) = e'(0, \rho) - h e(0, \rho)$,
$\Delta(\rho) = \det u(\rho)$. 
By property ($i_3$) of the Jost solution, the functions $u(\rho)$ and $\Delta(\rho)$ are
analytic for $\mbox{Im}\,\rho > 0$ and continuous for $\rho \in \Omega$.

Introduce the sets
$$
 	\Lambda = \{ \la = \rho^2 \colon \rho \in \Omega, \quad \Delta(\rho) = 0 \},
$$
$$
 	\Lambda' = \{ \la = \rho^2 \colon \mbox{Im}\,\rho > 0, \quad \Delta(\rho) = 0 \},
$$
$$
 	\Lambda'' = \{ \la = \rho^2 \colon \mbox{Im}\,\rho = 0, \quad \rho \ne 0, \quad \Delta(\rho) = 0 \}.
$$
It is known (see \cite{FY07}), that the spectrum of the boundary value problem $L$ consists of the positive half-line $\{ \la \colon \la \ge 0 \}$
and the discrete bounded set $\Lambda = \Lambda' \cup \Lambda''$. The set of all nonzero eigenvalues coincides with
the at most countable set $\Lambda'$. The points of $\Lambda''$ are called {\it spectral singularities } of $L$.

One can easily show that the Weyl solution and the Weyl matrix admit the following representations
\begin{equation} \label{reprPhie}
 	\Phi(x, \la) = e(x, \rho) (u(\rho))^{-1}, 	
\end{equation}
\begin{equation} \label{reprMe}
M(\la) = e(0, \rho) (u(\rho))^{-1}.
\end{equation}

Clearly, singularities of the Weyl matrix $M(\la)$ coincide with the zeros of $\Delta(\rho)$.

\begin{lem}[\cite{FY07, BF14}] \label{lem:M}
The Weyl matrix is analytic in $\Pi$ outside the countable bounded set of poles $\Lambda'$, and continuous in $\Pi_1$
outside the bounded set $\Lambda$.
For $|\rho| \to \iy$, $\rho \in \Omega$,
\begin{equation} \label{asymptM}
	M(\la) = \frac{1}{i\rho} \left( I_m + \frac{h}{i\rho} + \frac{\kappa(\rho)}{\rho}\right), \quad
	\kappa(\rho) = -i \int_0^{\iy} Q(t) e^{2 i \rho t} dt + O(\rho^{-2}).	  	
\end{equation}
\end{lem}

Let $\vv(x, \la) = [\vv_{jk}(x, \la)]_{j,k =1}^m$ and $S(x, \la) = [S_{jk}(x, \la)]_{j, k = 1}^m$ 
be the matrix solutions of equation~\eqref{eqv} under initial conditions 
$\vv(0, \la) = I_m$, $\vv'(0, \la) = h$, $S(0, \la) = 0_m$, $S'(0, \la) = I_m$.
For each fixed $x \ge 0$, these matrix functions are entire in $\la$-plane. 
Further we also need the following relation
\begin{equation} \label{reprPhiS}
\Phi(x, \la) = S(x, \la) + \vv(x, \la) M(\la).
\end{equation}

Symmetrically one can introduce the matrix solutions $\Phi^*(x, \la)$, $S^*(x, \la)$ and $\vv^*(x, \la)$
of equation \eqref{eqv*}, and the Weyl matrix $M^*(\la) := \Phi^*(0, \la)$ of the problem $L^*$.
Then
\begin{equation} \label{reprPhiS*}
\Phi^*(x, \la) = S^*(x, \la) + M^*(\la) \vv^*(x, \la).
\end{equation}

By virtue of \eqref{wron}, the expression $\langle \Phi^*(x, \la), \Phi(x, \la) \rangle$ does not depend
on $x$. Since by the boundary conditions
$$
 \langle \Phi^*(x, \la), \Phi(x, \la) \rangle_{x = 0} = U^*(\Phi^*) \Phi(0, \la) - \Phi^*(0, \la) U(\Phi) = M(\la) - M^*(\la), 
$$
$$
\lim\limits_{x \to \iy} \langle \Phi^*(x, \la), \Phi(x, \la) \rangle = 0_m, \quad \mbox{Im}\, \rho > 0, 
$$
we have $M(\la) \equiv M^*(\la)$.

Now proceed to the constructive solution of Inverse Problem~1. Let the Weyl matrix $M(\la)$ 
of the boundary value problem $L = L(Q, h)$ be given. Choose an arbitrary model problem
$\tilde L = L(\tilde Q, \tilde h)$ in the same form as $L$, but with other coefficients.
We agree that if a certain symbol $\gamma$ denotes an object related to $L$, then the corresponding symbol
$\tilde \gamma$ with tilde denotes the analogous object related to $\tilde L$.
We consider also the problem $\tilde L^* = L^*(\tilde Q, \tilde h)$. 

Denote 
$$
  M^{\pm}(\la) := \lim_{z \to 0, \, \mbox{Re}\, z > 0} M(\la \pm i z), \quad
  V(\la) := \frac{1}{2\pi i} (M^{-}(\la) - M^{+}(\la)), \quad \la > 0. 
$$
Suppose that the following condition is fulfilled:
\begin{equation} \label{restV}
	\int_{\rho^*}^{\iy} \rho^4 \| {\hat V}(\la) \|^2 d\rho < \iy, \quad \hat V := V - \tilde V,		                                                   		
\end{equation}
for some $\rho^* > 0$. For example, if $Q \in L_2((0, \iy); \mathbb{C}^{m \times m})$, then
$\kappa(\rho)$ in \eqref{asymptM} is $L_2$-function. Therefore one can take any problem $\tilde L$
with a potential from $L((0, \iy); \mathbb{C}^{m \times m}) \cap L_2((0, \iy); \mathbb{C}^{m \times m})$ and $\tilde h = h$, in order to satisfy \eqref{restV}.

Introduce auxiliary functions 
\begin{equation} \label{defD}
 	\tilde D(x, \la, \mu) = \frac{\langle \tilde \vv^*(x, \mu), \tilde \vv(x, \la)}{\la - \mu} =
 	\int_0^x \tilde \vv^*(t, \mu) \tilde \vv(t, \la) \, dt,
\end{equation}
$$
 	\hat M(\mu) = M(\mu) - \tilde M(\mu), \quad \tilde r(x, \la, \mu) = \hat M(\mu) \tilde D(x, \la, \mu).
$$

Let $\ga'$ be a bounded closed contour in $\la$-plane, encircling the set 
of singularities $\Lambda \cup \tilde \Lambda \cup \{ 0 \}$,
let $\ga''$ be the two-sided cut along the ray $\{ \la \colon \la > 0, \, \la \notin \mbox{int}\,\ga' \}$,
and $\ga = \ga' \cup \ga''$ be a contour with the counter-clockwise circuit (see Fig.~1). By contour integration over the contour $\ga$, 
G.~Freiling and V.~Yurko \cite{FY07} have obtained the following result. 

\begin{center}

\setlength{\unitlength}{1mm}
\begin{picture}(50, 40)
\linethickness{0.3mm}
\put(20,0){\vector(0,1){40}}
\put(0,20){\vector(1,0){50}}
\linethickness{0.15mm}
\put(45,22){\vector(-1,0){8}}
\put(31,22){\line(1,0){6}}
\put(31,18){\vector(1,0){7}}
\put(38,18){\line(1,0){7}}
\put(23,30.5){\vector(-2,1){0.5}}
\put(9,31){$\ga$}
\linethickness{0.075mm}
\qbezier(20,31)(10,30)(9,20)
\qbezier(9,20)(10,10)(20,9)
\qbezier(20,31)(30,30)(31,22)
\qbezier(20,9)(30,10)(31,18)

\end{picture}

\medskip

Fig. 1
\end{center}

\begin{thm} \label{thm:main}
For each fixed $x \ge 0$, the following relation holds
\begin{equation} \label{main}
 	\tilde \vv(x, \la) = \vv(x, \la) + \frac{1}{2 \pi i} \int_{\ga} \vv(x, \mu) \tilde r(x, \la, \mu) \, d\mu, 	
\end{equation}
which is called the main equation of Inverse Problem~1. This equation is uniquely solvable
(with respect to $\vv(x, \la)$) in the Banach space
$B$ of continuous bounded on $\ga$ matrix functions
$z(\la) = [z_{jk}(x, \la)]_{j, k = 1}^m$ with the norm
$\| z \|_B = \sup\limits_{\la \in \ga}\max\limits_{j, k = \overline{1, m}} |z_{jk}(\la)|$.
\end{thm}                    

\begin{cor}
The analogous relation is valid for $\Phi(x, \la)$:
\begin{equation} \label{mainPhi}
 	\tilde \Phi(x, \la) = \Phi(x, \la) + \frac{1}{2\pi i} \int_{\ga}
 	\vv(x, \mu)\hat M(\mu) \frac{\langle  \tilde \vv^*(x, \mu), \tilde \Phi(x, \la) \rangle}{\la - \mu}  d\mu, \quad \la \in J_{\ga},
\end{equation}
where $J_\ga := \{\la\colon \la \notin \ga \cup \mbox{int}\, \ga' \}$.
\end{cor}

\begin{proof}
Following the proof of Theorem~4.1 from \cite{FY07}, we define a block-matrix of spectral mappings 
$P(x, \la) = [P_{jk}(x, \la)]_{j, k = 1, 2}$ by the relation 
\begin{equation*}
 	P(x, \la) \begin{bmatrix} \tilde \vv(x, \la) & \tilde \Phi(x, \la) \\ \tilde \vv'(x, \la) & \tilde \Phi'(x, \la) \end{bmatrix} =
 	          \begin{bmatrix} \vv(x, \la) & \Phi(x, \la) \\ \vv'(x, \la) & \Phi'(x, \la) \end{bmatrix}.
\end{equation*}
In particular,
$$
 	\Phi(x, \la) = P_{11}(x, \la) \tilde \Phi(x, \la) + P_{12}(x, \la) \tilde \Phi'(x, \la).
$$ 
Substituting formulas (4.4) from \cite{FY07}:
$$
  	P_{1k}(x, \la) = \de_{1k} I_m + \frac{1}{2 \pi i} \int_{\ga} \frac{P_{1k}(x, \mu)}{\la - \mu} d\mu, \quad \la \in J_{\ga},
$$
where $\de_{jk}$ is the Kronecker delta, we get
\begin{equation} \label{smeqPhi}
 	\Phi(x, \la) = \tilde \Phi(x, \la) + \frac{1}{2 \pi i} \int_{\ga} \frac{P_{11}(x, \mu) \tilde \Phi(x, \la) + P_{12}(x, \mu) \tilde \Phi'(x, \la)}{\la - \mu} d\mu, \quad \la \in J_{\ga}.
\end{equation}
Note that the matrix functions $\Phi(x, \la)$ and $\tilde \Phi(x, \la)$ do not have singularities in $J_{\ga}$.

By virtue of the relations (3.12) from \cite{FY07},
$$
 \begin{array}{l}
 P_{11}(x, \mu) = \vv(x, \mu) \tilde {\Phi^*}'(x, \mu) - \Phi(x, \mu) \tilde {\vv^*}'(x, \mu), \\
 P_{12}(x, \mu) = \Phi(x, \mu) \tilde \vv^*(x, \mu) - \vv(x, \mu) \tilde \Phi^*(x, \mu).
 \end{array}
$$
Substitute these relations into \eqref{smeqPhi} and group the terms:
\begin{multline*}
\Phi(x, \la) = \tilde \Phi(x, \la) + \frac{1}{2\pi i} \int\limits_{\ga} \Biggl( \vv(x, \mu)
\left( \tilde {\Phi^*}'(x, \mu) \tilde \Phi(x, \la) - \tilde \Phi^*(x, \mu) \tilde \Phi'(x, \la) \right) \\ -
\Phi(x, \mu) \left( \tilde {\vv^*}'(x, \mu) \tilde \Phi(x, \la) - \tilde \vv^*(x, \mu) \tilde \Phi'(x, \la) \right) \Biggr)\frac{d \mu}{\la - \mu} \\ =
\tilde \Phi(x, \la) + \frac{1}{2 \pi i} \int\limits_{\ga} \frac{\vv(x, \mu) \langle \tilde \Phi^*(x, \mu), \tilde \Phi(x, \la)\rangle-
\Phi(x, \mu) \langle \tilde \vv^*(x, \mu), \tilde \Phi(x, \la) \rangle}{\la - \mu}d \mu.
\end{multline*}
If one expand $\Phi(x, \mu)$ and $\tilde \Phi^*(x, \mu)$, using \eqref{reprPhiS} and \eqref{reprPhiS*},
the terms with $S(x, \mu)$ and $\tilde S^*(x, \mu)$ will be analytic inside the contour and vanish by Cauchy's theorem.
Therefore we get
$$
 \Phi(x, \la) = \tilde \Phi(x, \la) + \frac{1}{2\pi i}\int\limits_{\ga} \frac{\vv(x, \mu) \tilde M^*(\mu) \langle \tilde \vv^*(x, \mu), \tilde \Phi(x, \la)\rangle
 - \vv(x, \mu) M(\mu) \langle \tilde \vv^*(x, \mu), \tilde \Phi(x, \la) \rangle}{\la - \mu}d\mu.
$$  
Since $\tilde M^*(\mu) \equiv \tilde M(\mu)$, we arrive at \eqref{mainPhi}.

\end{proof}

Solving the main equation \eqref{main}, one gets the matrix function $\vv(x, \la)$ and can follow the
algorithm from \cite{FY07} to recover the original problem $L$. But further we need an alternative 
way to construct the potential $Q$ and the coefficient $h$.

Let
\begin{equation} \label{defeps}
 	\eps_0(x) = \frac{1}{2 \pi i} \int_{\ga} \vv(x, \mu) \hat M(\mu) \tilde \vv^*(x, \mu) d \mu, \quad
 	\eps(x) = - 2 \eps_0'(x).
\end{equation}
Then similarly to \cite[Section 2.2]{FY01}, one can obtain the relations
\begin{equation} \label{Qh}
Q(x) = \tilde Q(x) + \eps(x), \quad h = \tilde h - \eps_0(0). 	
\end{equation}
Using the formulas \eqref{defeps}, \eqref{Qh}, one can construct $Q$ and $h$ by the solution of the main 
equation \eqref{main}, and solve Inverse Problem~1.

\bigskip

{\large \bf 3. Necessary and sufficient conditions}\\

In this section, we give the necessary and sufficient conditions in a very general form, 
with requirement of the solvability of the main equation.

Denote by $W$ the class of the matrix functions $M(\la)$, satisfying the conditions of Lemma~\ref{lem:M}, 
namely

($i_1$) $M(\la)$ is analytic in $\Pi$ outside the countable bounded set of poles $\Lambda'$, and continuous in $\Pi_1$
outside the bounded set $\Lambda$;

($i_2$) $M(\la)$ enjoys the asymptotic representation 
\begin{equation} \label{asymptM2}
	M(\la) = \frac{1}{i \rho} \left( I_m + \frac{h}{i \rho} + o(\rho^{-1}) \right), \quad |\rho| \to \iy, \quad \rho \in \Omega.
\end{equation} 

\begin{thm} \label{thm:NSC}
For the matrix function $M(\la) \in W$ to be the Weyl matrix of some boundary value problem $L$
of the form \eqref{eqv}, \eqref{IC}, it is necessary and sufficient to satisfy the following conditions.
\begin{enumerate}
\item There exists a model problem $\tilde L$, such that \eqref{restV} holds.
\item For each fixed $x \ge 0$, the main equation \eqref{main} is uniquely solvable.
\item $\eps(x) \in L((0, \iy); \mathbb{C}^{m \times m})$, where $\eps(x)$ was defined in \eqref{defeps}.
\end{enumerate}
\end{thm}
 
Similarly one can study the classes of potentials $Q$ with higher degree of smotheness, 
then the potential of the model problem $\tilde Q$ and $\eps$ should belong to the same classes.

\begin{proof}
By necessity, Conditions 1 and 3 are obvious, while Condition~2 is contained in Theorem~\ref{thm:main}.
So it remains to prove, that the potential $Q$ and the coefficient $h$, constructed by formulas \eqref{Qh},
form a problem $L$ with the Weyl matrix, coinciding with the given $M(\la)$.

{\bf Step 1.} Let $M \in W$, the model problem $\tilde L$ satisfy the Condition~1, $\vv(x, \la)$ be 
the solution of the main equation \eqref{main},
 and $Q$, $h$ be constructed via \eqref{Qh}.
Let us prove that 
\begin{equation} \label{eqvv}
\ell \vv(x, \la) = \la \vv(x, \la),
\end{equation}
where the differential expression $\ell$ was defined in \eqref{eqv}.

Differentiating \eqref{main} and using \eqref{defD}, we get
\begin{multline*}
\ell \tilde \vv(x, \la) = \ell \vv(x, \la) + \frac{1}{2 \pi i} \int_{\ga} \ell \vv(x, \mu)
\tilde r(x, \la, \mu) \, d\mu \\ - \frac{1}{2 \pi i} \int_{\ga} \Bigl( 2 \vv'(x, \mu) \hat M(\mu) 
\tilde \vv^*(x, \mu) \tilde \vv(x, \la) \, d\mu + \vv(x, \mu) \hat M(\mu) 
\bigl( \tilde \vv^*(x, \mu) \tilde \vv(x, \la)\bigr)' \Bigr) \, d\mu.
\end{multline*}
Since by \eqref{Qh}
$$
 	Q(x) = \tilde Q(x) - 2 \eps_0'(x) = \tilde Q(x) - \frac{1}{2 \pi i} 
 	\int_{\ga} \Bigl( 2 \vv'(x, \mu) \hat M(\mu) \tilde \vv^*(x, \mu) + 2 \vv(x, \mu) \hat M(\mu)  \tilde {\vv^*}'(x,\mu)  \Bigr)\, d\mu,
$$
we obtain
\begin{equation*}
\tilde \ell \tilde \vv(x, \la) = \ell \vv(x, \la) + \frac{1}{2 \pi i} \int_{\ga} \ell \vv(x, \mu) \tilde r(x, \la, \mu) \, d\mu +
\frac{1}{2\pi i} \int_{\ga} \vv(x, \mu) \hat M(\mu) \langle \tilde \vv^*(x, \mu), \tilde \vv(x, \la) \rangle d\mu.
\end{equation*}
Taking \eqref{defD} and the relation $\tilde \ell \tilde \vv = \la \tilde \vv$ into account, we conclude that
$$
 	\la \tilde \vv(x, \la) = \ell \vv(x, \la) + \frac{1}{2 \pi i} \int_{\ga} \ell \vv(x, \mu) \tilde r(x, \la, \mu) \, d\mu +
 	\frac{1}{2 \pi i} \int_{\ga} (\la - \mu) \vv(x, \mu) \tilde r(x, \la, \mu). 
$$
Substituting \eqref{main} into this relation, we arrive at the equation
\begin{equation} \label{eqeta}
 	\eta(x, \la) + \frac{1}{2 \pi i} \int_{\ga} \eta(x, \mu) \tilde r(x, \la, \mu) \, d\mu = 0_m, \quad \la \in \ga,
\end{equation}
with respect to $\eta(x, \la) = \ell \vv(x, \la) - \la \vv(x, \la)$.

Suppose $\int\limits_{\la^*}^{\iy} \rho \| \hat V(\la) \| d\la < \iy$ (in the general case, under assumption \eqref{restV}, we have
$\int\limits_{\la^*}^{\iy} \| \hat V(\la) \| \, d\la < \iy$, $\la^* = (\rho^*)^2$). Then, using the same arguments, as in the scalar case
\cite{FY01}, one can show, that the matrix function $\eta(x, \la)$ belongs to the Banach space $B$ for each fixed $x \ge 0$.
Consider the operator $\tilde R(x) \colon B \to B$, acting in the following way:
$$
 	z(\la) \tilde R(x) = \frac{1}{2 \pi i} \int_{\ga} z(\mu) \tilde r(x, \la, \mu) \, d\mu. 
$$
Here and below in similar situations, we write an operator to the right of an operand,
because the action of the operator involves noncommutative matrix multiplication in the such order.
For each fixed $x \ge 0$, the operator $\tilde R(x)$ is compact, therefore it follows from the unique solvability of the main equation \eqref{main}, that the corresponding
homogeneous equation \eqref{eqeta} is also uniquely solvable. Hence $\eta(x, \la) \equiv 0$, and \eqref{eqvv} is proved.

\smallskip

{\bf Step 2.} In the general case, when \eqref{restV} holds, the proof of the equality \eqref{eqvv} is more complicated,
so we only outline the main ideas. Introduce contours $\ga_N = \ga \cap \{ |\la| \le N^2 \}$, and consider operators
$$
 	\tilde R_N(x) \colon B \to B, \quad z(\la) \tilde R_N(x) = \frac{1}{2 \pi i} \int_{\ga_N} z(\mu) \tilde r(x, \la, \mu)\, d\mu.
$$  
The sequence $\{ \tilde R_N(x) \}$ converges to $\tilde R(x)$ in the operator norm. In view of the unique solvability of the main equation,
the operator $(I + \tilde R(x))$ is invertible for each fixed $x \ge 0$. So for sufficiently large values of $N$,
the operators $(I + \tilde R_N(x))$ are also invertible, and the equations
$$
 	\tilde \vv(x, \la) = \vv_N(x, \la) + \frac{1}{2 \pi i} \int_{\ga_N} \vv_N(x, \mu) \tilde r(x, \la, \mu) \,d\mu
$$
have unique solutions $\vv_N(x, \la)$. Since $\int\limits_{\rho^*}^N \rho \| \hat V(\la)\| d\la < \iy$, one can repeat the
arguments of Step~1 for the matrix functions $\vv_N(x, \la)$, and prove the relations
$-\vv_N''(x, \la) + Q_N(x) \vv_N(x, \la) = \la \vv_N(x, \la)$ with matrix potentials $Q_N(x) = Q(x) + \eps_N(x)$, where
$$
 	\eps_N(x) = -2 \left( \frac{1}{2 \pi i} \int_{\ga_N} \vv_N(x, \mu) \hat M(\mu) \tilde \vv^*(x, \mu) \, d\mu\right)'.
$$
The sequence $\{ \vv_N(x, \la) \}$ converges to $\vv(x, \la)$ uniformly with respect to $x$ and $\la$ on compact sets,
and the sequence $\{ Q_N(x) \}$ converges to $Q(x)$ in $L$-norm on every bounded interval. These facts yield \eqref{eqvv}.

Analogously one can prove the relation $\ell \Phi(x, \la) = \la \Phi(x, \la)$
for the matrix function $\Phi(x, \la)$, constructed via \eqref{mainPhi}.

\smallskip

{\bf Step 3.} Substituting $x = 0$ into the main equation \eqref{main}, we get 
$\vv(0, \la) = I_m$. Differentiate the main equation:
$$
	\tilde \vv'(x, \la) = \vv'(x, \la) + \frac{1}{2\pi i} \int_{\ga} \vv'(x, \mu) \tilde r(x, \la, \mu) \, d\mu + 
	\frac{1}{2 \pi i} \int_{\ga} \vv(x, \mu) \hat M(\mu) \tilde \vv^*(x, \mu) \tilde \vv(x, \la) \, d\mu.
$$
Taking \eqref{defD}, \eqref{defeps} and \eqref{Qh} into account, we obtain
$$
 	\vv'(0, \la) = \tilde \vv'(0, \la) - \frac{1}{2 \pi i} \int_{\ga} \vv(0, \mu) \hat M(\mu) \tilde \vv^*(0, \mu) \, d\mu = \tilde h - \eps_0(0) = h.
$$
Similarly using \eqref{mainPhi}, one can check that $U(\Phi) = I_m$.

The following standard estimates are valid for $\nu = 0, 1$:
$$
 	\| \vv^{(\nu)}(x, \mu)\|, \, \| \vv^{*(\nu)}(x, \mu) \| \le C |\theta|^{\nu}, \quad \mu = \theta^2 \in \ga, 
$$
$$
 	\|\tilde \Phi^{(\nu)}(x, \la) \| \le C |\rho|^{\nu - 1} \exp(-|\tau|x), \quad \la \notin \tilde \Lambda \cup \{ 0 \}.
$$
In view of \eqref{asymptM2}, $\hat M(\mu) = O(\mu^{-1})$, $|\mu| \to \iy$, $\mu \in \Pi$. Taking $\la \notin \mbox{int} \ga$
and substituting these estimates into \eqref{mainPhi}, we derive
$$
 	\| \Phi(x, \la) \exp(- i \rho x) \| \le C \left( 1 + \int_{\la^*}^{\iy} \frac{d\mu}{\mu |\la - \mu|} \right) \le C_1.
$$ 
Thus, we have $\Phi(x, \la) = O(\exp(i \rho x))$, so $\Phi(x, \la)$, constructed via \eqref{mainPhi}, is the Weyl solution 
of the problem $L(Q, h)$.

It follows from \eqref{mainPhi}, that
$$
 	\Phi(0, \la) = \tilde M(\la) + \frac{1}{2 \pi i} \int_{\ga} \frac{\hat M(\la)}{\la - \mu}d\mu.
$$
Using Cauchy integral formula, it is easy to show that 
$$
 	\hat M(\la) = \frac{1}{2\pi i} \int_{\ga} \frac{\hat M(\mu)}{\la - \mu} d\mu.
$$
Consequently, $\Phi(0, \la) = M(\la)$, i.e. the given matrix $M(\la)$ is the Weyl matrix of the constructed 
boundary value problem $L(Q, h)$.
\end{proof}

\bigskip

{\large \bf 4. Self-adjoint case: properties of the spectral data}\\

In this section, we assume that the boundary value problem $L$ is self-adjoint: $Q(x) = Q^{\dagger}(x)$ a.e. on $(0, \iy)$, $h = h^{\dagger}$.
We show that its spectrum has the following properties ($i_1$)-($i_6$). The similar facts for the Dirichlet boundary condition were
proved in \cite{AM60}. 

{\bf Property ($i_1$).} {\it The problem $L$ does not have spectral singularities: $\Lambda'' = \varnothing$.}

\begin{proof}
We have to prove that $\det u(\rho) \ne 0$ for $\rho \in \mathbb{R} \backslash \{ 0 \}$.
In view of \eqref{inteqe} and \eqref{inteqe*}, 
$$
 	e^{\dagger}(x, \rho) = \exp(- i\rho x) \, I_m + \frac{1}{2 i \rho} \int_x^{\iy} \left(\exp(-i \rho (x - t) ) -
 	\exp(i \rho (x - t))\right) e^{\dagger}(t, \rho) Q(t) \, dt = e^*(x, - \rho) 
$$
for $\rho \in \mathbb{R} \backslash \{ 0 \}$,
therefore $u^{\dagger}(\rho) = u^*(-\rho)$ for such $\rho$. Suppose there exist a real $\rho_0 \ne 0$
and a nonzero vector $a$, such that $u(\rho_0) a = 0$ and, consequently, $a^{\dagger} u^*(- \rho_0) = 0$.
On the one hand,
$$
 	a^{\dagger} \langle e^*(x, -\rho_0), e(x, \rho_0) \rangle a = [ a^{\dagger} u^*(-\rho_0) ] e(0, \rho_0) a -
 	a^{\dagger} e^*(0, -\rho_0) [u(\rho_0) a] = 0.
$$
On the other hand, using \eqref{wronee}, we obtain
$$
a^{\dagger} \langle e^*(x, -\rho_0), e(x, \rho_0) \rangle a = -2 i\rho_0 a^{\dagger} a \ne 0.
$$
So we arrive at the contradiction, which proves the property. 
\end{proof}

{\bf Property ($i_2$).} {\it All the nonzero eigenvalues are real and negative: $\la_k = \rho_k^2 < 0$, $\rho_k = i \tau_k$, $\tau_k > 0$.}

Indeed, the eigenvalues of $L$ are real because of the self-adjointness. In view of \cite[Theorem 2.4]{FY07}, they cannot be positive.

{\bf Property ($i_3$).} {\it The poles of the matrix function $(u(\rho))^{-1}$ in the upper half-plane are simple.
(They coincide with $i \tau_k$).}

\begin{proof}
Start from the relations
\begin{equation} \label{eqve}
-e''(x, \rho) + Q(x) e(x, \rho) = \rho^2 e(x, \rho),
\end{equation}
\begin{equation} \label{eqve*}
-{e^*}''(x, \rho) + e^*(x, \rho) Q(x) = \rho^2 e^*(x, \rho).
\end{equation}
Differentiate \eqref{eqve*} by $\rho$, multiply it by $e(x, \rho)$ and subtract \eqref{eqve}, multiplied by $\frac{d}{d\rho}e^*(x, \rho)$
from the left:
\begin{equation*} 
 	\frac{d}{d \rho} e^*(x, \rho) e''(x, \rho) - \frac{d}{d\rho} {e^*}''(x, \rho) e(x, \rho) = 2 \rho\, e^*(x, \rho) e(x, \rho).
\end{equation*}
Note the the left-hand side of this relation equals $-\frac{d}{dx}\langle \frac{d}{d\rho}e^*(x, \rho), e(x, \rho)\rangle$.
Integrating by $x$ from $0$ to $\iy$, we get
$$
 \langle \frac{d}{d\rho}e^*(x, \rho), e(x, \rho)\rangle \bigg|_0^{\iy} = 2 \rho \int_0^{\iy} e^*(x, \rho) e(x, \rho) \, dx.	
$$
Let $\mbox{Im}\, \rho > 0$. Then the matrix functions $e(x, \rho)$, $\frac{d}{d\rho}e^*(x, \rho)$ and their
$x$-derivatives tend to zero as $x \to \iy$. Consequently, we obtain
\begin{equation} \label{smeqwron2}
 \frac{d}{d\rho}u^*(\rho) e(0, \rho) - \frac{d}{d\rho}e^*(0, \rho) u(\rho) = 2 \rho \int_0^{\iy} e^*(x, \rho) e(x, \rho) \, dx.
\end{equation}

Let $\rho$ be equal to $\rho_0 = \sqrt \la_0$, where $\la_0$ is an eigenvalue, and $u(\rho_0) a = 0$, $a \ne 0$. For purely imaginary $\rho$, one has
$e^*(x, \rho) = e^{\dagger}(x, \rho)$ and $u^*(\rho) = u^{\dagger}(\rho)$. So we derive from \eqref{smeqwron2}:
\begin{equation} \label{smeqwron3}
 	-a^{\dagger} \frac{d}{d\rho}u^{\dagger}(\rho_0) e(0, \rho_0) a = 2 \rho_0 a^{\dagger} \int_0^{\iy} e^{\dagger}(x, \rho_0) e(x, \rho_0) \, dx \, a \ne 0.
\end{equation}

In order to prove the simplicity of the poles for $(u(\rho))^{-1}$, we adapt Lemma~2.2.1 from \cite{AM60}:
\smallskip
{\it
The inverse $(u(\rho))^{-1}$ has a simple pole at $\rho = \rho_0$, if and only if the relations
\begin{equation}  \label{AMab}
u(\rho_0) a = 0, \quad u(\rho_0) b + \frac{d}{d\rho} u(\rho_0) a = 0,
\end{equation}
where $a$ and $b$ are constant vectors, yield $a = 0$.
}
\smallskip

Let vectors $a$ and $b$ satisfy \eqref{AMab}. Then 
$$
-a^{\dagger} \frac{d}{d\rho}u^{\dagger}(\rho_0) e(0, \rho_0) a = b^{\dagger} u^{\dagger}(\rho_0) e(0, \rho_0) a. 
$$
Since
$$
  \langle e^*(x, \rho), e(x, \rho) \rangle = \langle e^*(x, \rho), e(x, \rho) \rangle_{x = \iy} = 0, \quad \mbox{Im}\,\rho > 0,
$$
one has
$$
   b^{\dagger} u^{\dagger}(\rho_0) e(0, \rho_0) a = b^{\dagger} e^*(0, \rho_0) u(\rho_0) a = 0,
$$
but this contradicts \eqref{smeqwron3}. So $a = 0$, and a square root of every eigenvalue $\rho = \rho_0$ is a simple pole of $(u(\rho))^{-1}$.
\end{proof}

The next properties take place, if the additional condition holds:
\begin{equation} \label{xQx}
\int_0^{\iy} x \| Q(x) \| \, dx < \iy.
\end{equation}

{\bf Property ($i_4$).} {\it The number of eigenvalues is finite.}

\begin{proof}
Prove the assertion by contradiction. Suppose there is an infinite sequence $\{ \la_k \}_{k = 1}^{\iy}$ of negative eigenvalues,
$\rho_k = \sqrt \la_k$, and $\{ Y_k(x) \}_{k = 1}^{\iy}$ is an ortogonal sequence of corresponding vector eigenfunctions. Note that
there can be multiple eigenvalues, their multiplicities are finite and equal to $m - \mbox{rank}\, u(\rho_k)$. Multiple eigenvalues
occur in the sequence $\{ \la_k\}_{k = 1}^{\iy}$ multiple times with different eigenfunctions $Y_k(x)$. The eigenfunctions
can be represented in the form $Y_k(x) = e(x, \rho_k) N_k$, $\| N_k \| = 1$. 

Using the ortogonality of the eigenfunctions, we obtain
for $k \ne n$
\begin{multline} \label{smeqintY}
0 = \int_0^{\iy} Y_k^{\dagger}(x) Y_n(x) \, dx =  N_k^{\dagger}\int_A^{\iy}	e^{\dagger}(x, \rho_k) e(x, \rho_n) \, dx \, N_n +
\int_0^{A} Y_k^{\dagger}(x)	 Y_k(x) \, dx \\
+ \int_0^A Y^{\dagger}_k(x) (Y_n(x) - Y_k(x)) \, dx =: \mathcal{I}_1 + \mathcal{I}_2 + \mathcal{I}_3.
\end{multline}

Similarly to the scalar case \cite[Theorem 2.3.4]{FY01}, one can show that $e(x, \rho_k) = \exp(-\tau_k x) (I_m + \al_k(x))$,
where $\| \al_k(x) \| \le \frac{1}{8}$ as $x \ge A$ for all $k \ge 1$ and for sufficiently large $A$. Consequently
\begin{multline*} 
 \mathcal{I}_1 = N_k^{\dagger} \int_A^{\iy} \exp(-(\tau_k + \tau_n)x) (I_m + \beta_{kn}(x)) \, dx \, N_k \\ + 
N_k^{\dagger} \int_A^{\iy} \exp(-(\tau_k + \tau_n)x) (I_m + \beta_{kn}(x)) \, dx \, (N_n - N_k), \quad \| \beta_{kn}(x) \| \le \frac{3}{8}. 
\end{multline*}
Since the vectors $N_k$ belong to the unit sphere, one can choose a convergent subsequence $\{ N_{k_s} \}_{s = 1}^{\iy}$.
Further we consider $N_k$ and $N_n$ from such subsequence. Then for sufficiently large $k$ and $n$, we have
\begin{multline*}
 \left|  N_k^{\dagger} \int_A^{\iy} \exp(-(\tau_k + \tau_n)x) (I_m + \beta_{kn}(x)) \, dx \, (N_n - N_k) \right| \\ \le
 \frac{3 \exp(-(\tau_k + \tau_n) A)}{2 (\tau_k + \tau_n)} \| N_n - N_k \| \le \frac{\exp(-(\tau_k + \tau_n) A)}{8 (\tau_k + \tau_n)}.
\end{multline*} 
Hence
$$
 	\mathcal{I}_1 \ge \frac{\exp(-(\tau_k + \tau_n) A)}{2 (\tau_k + \tau_n)}  \ge \frac{\exp(-2AT)}{4T}, \quad T := \max\limits_k \tau_k.
$$

Clearly, $\mathcal{I}_2 \ge 0$.
Using arguments similar to the proof of \cite[Theorem 2.3.4]{FY01}, one can show that $\mathcal{I}_3$ tends to zero, as $k, n \to \iy$.
Thus, for sufficiently large $k$ and $n$, $\mathcal{I}_1 + \mathcal{I}_2 + \mathcal{I}_3 > 0$, that contradicts \eqref{smeqintY}.
Hence, the number of negative eigenvalues is finite.
\end{proof}

{\bf Property ($i_5$).} {\it $\la = 0$ is not an eigenvalue of $L$.}

\begin{proof}
It was proved in \cite{AM60}, that if condition \eqref{xQx} holds, the Jost solution $e(x, \rho)$ exists for $\rho = 0$.
So equation \eqref{eqv} for $\la = 0$ has the solution $e(x, 0) = I_m + o(1)$, as $x \to \iy$. One can easily check that
the matrix function
$$
 	z(x) = e(x, 0) \int_0^x \left(e^*(t, 0) e(t, 0) \right)^{-1} dt
$$ 
is also a solution of \eqref{eqv} for $\la = 0$, and it enjoys asymptotic representation $z(x) = x (I_m + o(1))$, as $x \to \iy$.
Thus, the columns of the matrices $e(x, 0)$ and $z(x)$ form a fundamental system of solutions for equation \eqref{eqv} for $\la = 0$.
If $\la = 0$ is an eigenvalue of $L$, then the corresponding vector eigenfunction should have an expansion
$Y_0(x) = e(x, 0) a + z(x) b$, where $a$ and $b$ are some constant vectors. But in view of asymptotic behavior of $e(x, 0)$
and $z(x)$, one has $\lim\limits_{x \to \iy} Y_0(x) = 0$ if and only if $a = b = 0$. So $\la = 0$ is not an eigenvalue of $L$.
\end{proof}

{\bf Property ($i_6$).} {\it $\rho (u(\rho))^{-1} = O(1)$ and $M(\la) = O(\rho^{-1})$ as $\rho \to 0$, $\rho \in \Omega$.}

\begin{proof}
Consider the matrix function 
$g(\rho) = 2 i \rho (u(\rho))^{-1}$.
It follows from \eqref{wronee}, that
$$
 	u^*(-\rho) e(0, \rho) - e^*(0, -\rho) u(\rho) = -2 i \rho I_m.
$$
Hence
$$
 	g(\rho) = e^*(0, -\rho) - u^*(-\rho) e(0, \rho) (u(\rho))^{-1}.
$$
In view of \eqref{reprMe} and the equality $M(\la) \equiv M^*(\la)$, one has 
$$
 	e(0, \rho) (u(\rho))^{-1} = M(\la) = (u^*(\rho))^{-1} e^*(0, \rho).
$$
Consequently, 
\begin{equation} \label{smeqg}
 	g(\rho) = e^*(0, -\rho) - \xi(\rho) e^*(0, \rho), \quad \xi(\rho) := u^*(-\rho) (u^*(\rho))^{-1}.
\end{equation}

Let $\rho \in \mathbb{R} \backslash \{ 0 \}$.
Expand the solution $\vv(x, \la)$ by the fundamental system of solutions with some matrix coefficients $A(\rho)$ and $B(\rho)$:
\begin{align}
\label{expandphi}
\vv(x, \la) & = e(x, \rho) A(\rho) + e(x, -\rho) B(\rho), \\
\label{expandphid}
\vv'(x, \la) & = e'(x, \rho) A(\rho) + e'(x, -\rho) B(\rho).
\end{align}
Multiplying \eqref{expandphi} by ${e^*}'(x, -\rho)$ and \eqref{expandphid} by $e^*(x, \rho)$ from the left and using \eqref{wronee},
one can derive
\begin{align*}
 	A(\rho) & = -\frac{1}{2 i \rho} \left( {e^*}'(x, -\rho)\vv(x, \la) - e^*(x, -\rho) \vv'(x, \la) \right), \\
 	B(\rho) & = \frac{1}{2 i \rho} \left( {e^*}'(x, \rho) \vv(x, \la) - e^*(x, \rho) \vv'(x, \la) \right).
\end{align*}
Since $A(\rho)$ and $B(\rho)$ do not depend on $x$, one can take $x = 0$ and obtain
$$
 	A(\rho) = -\frac{1}{2\pi i} u^*(-\rho), \quad B(\rho) = \frac{1}{2 i \rho} u^*(\rho).
$$
Finally 
$$
 	\vv(x, \la) = -\frac{1}{2 i \rho} \left( e(x, \rho) u^*(-\rho) - e(x, -\rho) u^*(\rho) \right).
$$
Since $U(\vv) = 0$, we get
$$
 	u(\rho) u^*(-\rho) = u(-\rho) u^*(\rho).
$$
Therefore
$$
 	\xi(\rho) = u^*(-\rho) (u^*(\rho))^{-1} = (u(\rho))^{-1} u(-\rho).
$$  
One can easily show that $u^*(\rho) = u^{\dagger}(-\rho)$ for real $\rho$ and, consequently, $\xi(\rho)$
is a unitary matrix for $\rho \in \mathbb{R} \backslash \{ 0 \}$. Then it follows from \eqref{smeqg},
that the matrix function $g(\rho)$ is bounded for $\rho \in \mathbb{R} \backslash \{ 0 \}$.

Consider the region $\mathcal{D} := \{ \rho \colon \mbox{Im}\, \rho > 0, |\rho| < \tau^* \}$, where $\tau^*$
is a number less than all $\tau_k$ (by property ($i_4$), there is a finite number of them).
Obviously, $g(\rho)$ is analytic in $\mathcal{D}$ and continuous in $\overline{\mathcal{D}} \backslash \{ 0 \}$.
If it is also analytic in zero, $\rho = 0$ is a removable singularity. Then $g(\rho)$ is continuous in $\overline{\mathcal{D}}$,
so $g(\rho) = O(1)$. In the general case, one can approximate the potential $Q(x)$ by the sequence of potentials
$$
   Q_{\beta}(x) = \begin{cases}
                Q(x), \quad 0 \le x \le \beta \\
                0, \quad x > \beta,
   			\end{cases}
$$
and use the technique from \cite{AM60} (see Lemma~2.4.1). 

Since under condition \eqref{xQx} the Jost solution exists for $\rho = 0$, 
we have $e(0, \rho) = O(1)$ as $\rho \to 0$. Taking \eqref{reprMe} and $g(\rho) = O(1)$ into account, we arrive at $M(\la) = O(\rho^{-1})$,
$\rho \to 0$.
\end{proof}

We combine the properties of the Weyl matrix in the next theorem.

\begin{thm}
Let $L = L(Q, h)$, $Q = Q^{\dagger}$, $h = h^{\dagger}$, $Q \in L((0, \iy); \mathbb{C}^{m \times m})$,
and the condition \eqref{xQx} holds. Then the Weyl matrix of this problem $M(\la)$ is analytic in $\Pi$
outside the finite set of simple poles $\Lambda' = \{ \la_k \}_{k = 1}^P$, $\la_k = \rho_k^2 < 0$, and continuous in $\Pi_1\backslash \Lambda$.
Moreover, 
$$
 	\al_k := \Res_{\la = \la_k} M(\la) \ge 0, \quad k = \overline{1, P},
$$
$$
 	M(\la) = O(\rho^{-1}), \quad \rho \to 0.
$$
The matrix function $\rho V(\la)$ is continuous and bounded for $\la > 0$ and $V(\la) > 0$ for $\la > 0$.
\end{thm}

\begin{proof}
Fix an eigenvalue $\la_k$, $k = \overline{1, P}$.
Cosider two representations \eqref{reprPhie} and \eqref{reprPhiS} of $\Phi(x, \la)$, and take for both of them
the residue with respect to the pole $\la_k$. Then we obtain the relation
\begin{equation} \label{eigenf}
\vv(x, \la_k) \al_k = e(x, \rho_k) u_k, \quad u_k := 2 \rho_k \Res_{\rho = \rho_k} (u(\rho))^{-1}.
\end{equation}
Note that the columns of the left-hand side and the right-hand side of \eqref{eigenf} are vector
eigenfunctions, corresponding to the eigenvalue $\la_k$.	

Further we consider $\rho$, such that $\mbox{Re} \, \rho = 0$, $\mbox{Im}\, \rho > 0$.
It is easy to check, that
$$
 	\langle e^*(x, \rho_k), e(x, \rho) \rangle_{x = \iy} - \langle e^*(x, \rho_k), e(x, \rho) \rangle_{x = 0} = 
 	(\la - \la_k) \int_0^{\iy} e^*(x, \rho_k) e(x, \rho) dx.
$$
Using asymptotics \eqref{asymptex} for $e(x, \rho)$ and $e^*(x, \rho_k)$, we get
$$
 	\lim_{\la \to \la_k} \frac{1}{\la - \la_k} \langle e^*(x, \rho_k), e(x, \rho) \rangle_{x = \iy} = 0_m.
$$
By virtue of the self-adjointness, $e^*(x, \rho) = e^{\dagger}(x, \rho)$, $\vv^*(x, \la) = \vv^{\dagger}(x, \la)$, $\la < 0$.
Therefore
$$
 \mathcal{S} := u_k^{\dagger} \int_0^{\iy} e^{\dagger}(x, \rho_k) e(x, \rho_k) \, dx u_k = - \lim_{\la \to \la_0}
 \frac{u_k^{\dagger} \langle e^{\dagger}(x, \rho_k), e(x, \rho) \rangle u_k}{\la - \la_0}.   	
$$
Substituting \eqref{eigenf}, we obtain
\begin{multline*}
 	\mathcal{S} = - \lim_{\la \to \la_k} \frac{\al_k^{\dagger} \left({\vv^{\dagger}}'(0, \la_k) e(0, \rho) -
 	\vv^{\dagger}(0, \la_k) e'(0, \rho) \right)}{\la - \la_k} \cdot \lim_{\la \to \la_k}(\la - \la_k) (u(\rho))^{-1} \\ = -\al_k^{\dagger}
 	\lim_{\rho \to \rho_k} (h e(0, \rho) - e'(0, \rho)) (u(\rho))^{-1} = \al_k^{\dagger}.
\end{multline*}
Obviously, $\mathcal{S} = \mathcal{S}^{\dagger} \ge 0$. Hence $\al_k = \al_k^{\dagger} \ge 0$.

Now consider $V(\la) = \frac{1}{2 \pi i} (M^-(\la) - M^+(\la))$, $\la > 0$.
Taking the relations \eqref{reprMe} and $M(\la) = M^*(\la)$ into account, we have
$$
 	M^-(\la) = (u^*(-\rho))^{-1} e^*(0, -\rho), \quad M^+(\la) = e(0, \rho) (u(\rho))^{-1}, \quad \rho > 0.
$$
Consequently
$$
 	V(\la) = -\frac{1}{2 \pi i} (u^*(-\rho))^{-1} \langle e^*(x, -\rho), e(x, \rho) \rangle (u(\rho))^{-1}.
$$
Substituting \eqref{wronee}, we get
$$
 	V(\la) = \frac{\rho}{\pi} (u^*(-\rho))^{-1} (u(\rho))^{-1}.
$$
For real values of $\rho$, one has $e^*(x, -\rho) = e^{\dagger}(x, \rho)$, $u^*(-\rho) = u^{\dagger}(\rho)$.
Since in the self-adjoint case the set of spectral singularities $\Lambda''$ is empty, $\det u(\rho) \ne 0$, $\rho \in \mathbb{R} \backslash \{ 0 \}$.
Hence $V(\la) = V^{\dagger}(\la) > 0$.

The remaining assertions of the theorem do not need a proof.
\end{proof}

We call the collection $\left( \{ V(\la)\}_{\la > 0}, \: \{ \la_k, \al_k \}_{k = 1}^P \right)$
{\it the spectral data} of $L$. Similarly to the scalar case (see \cite{FY01}), 
the Weyl matrix can be uniquely determined by the spectral data:
\begin{equation} \label{Msd}
 	M(\la) =  \int_0^{\iy} \frac{V(\mu)}{\la - \mu} d\mu + \sum_{k = 1}^P \frac{\al_k}{\la - \la_k}, \quad \la \in \Pi \backslash \Lambda'.
\end{equation}

{\large \bf 5. Self-adjoint case: the inverse problem}\\

Now we are going to apply the general results of Section~3 to the self-adjoint case.

Let us rewrite the main equation \eqref{main} of the inverse problem in terms of the spectral data.
Denote
\begin{gather*}
 	\la_{n0} = \la_n, \quad \la_{n1} = \tilde \la_n, \quad
 	\al_{n0} = \al_n, \quad \al_{n1} = \tilde \al_n, \\
 	\vv_{ni}(x) = \vv(x, \la_{ni}), \quad \tilde \vv_{ni}(x) = \tilde \vv(x, \la_{ni}), \quad
 	\mathcal{M} = \{ (n, 0) \}_{n = 1}^P \cup \{ (n, 1) \}_{n = 1}^{\tilde P}.
\end{gather*}
Then the main equation \eqref{main} can be transformed into the system of equations
\begin{multline} \label{maincont}
\tilde \vv(x, \la) = \vv(x, \la) + \int_0^{\iy} \vv(x, \mu) \hat V(\mu) \tilde D(x, \la, \mu)\, d\mu \\ + 
\sum_{k = 1}^P \vv_{k0}(x) \al_{k0} \tilde D(x, \la, \la_{k0}) - \sum_{k = 1}^{\tilde P} \vv_{k1}(x) \al_{k1} \tilde D(x, \la, \la_{k1}), \quad
\la > 0, 
\end{multline}
\begin{multline} \label{maindisc}
\tilde \vv_{ni}(x) = \vv_{ni}(x) + \int_0^{\iy} \vv(x, \mu) \hat V(\mu) \tilde D(x, \la_{ni}, \mu) \, d\mu \\ +
\sum_{k = 1}^P \vv_{k0}(x) \al_{k0} \tilde D(x, \la_{ni}, \la_{k0}) - \sum_{k = 1}^{\tilde P} \vv_{k1}(x) \al_{k1} \tilde D(x, \la_{ni}, \la_{k1}),
\quad (n, i) \in \mathcal{M},
\end{multline}
with respect to the element $\psi(x) := \left(\{ \vv(x, \la) \}_{\la > 0}, \{ \vv_{ni}(x) \}_{(n, i) \in \mathcal{M}}\right)$ of the Banach space
$B_S$ of pairs $(F_1, F_2)$, where 
$$
	F_1 \in C((0, \iy); \mathbb{C}^{m \times m}), \quad F_2 = \{ f_{ni} \}_{(n, i) \in \mathcal{M}}, \quad
 	f_{ni} \in \mathbb{C}^{m \times m},
$$
with the norm
$$
 	\| (F_1, F_2) \|_{B_S} = \max\left( \sup_{\la > 0} \| F_1(\la) \|, \max\limits_{(n, i) \in \mathcal{M}} \| f_{ni} \| \right).
$$
The system \eqref{maincont}-\eqref{maindisc} has the form $\psi(x) (I + \tilde R(x)) = \tilde \psi(x)$, where
$\tilde R(x) \colon B_S \to B_S$ is a linear compact operator for each fixed $x \ge 0$. By necessity, 
we have the unique solvability of the main equation \eqref{main}, so the equivalent system \eqref{maincont}-\eqref{maindisc}
is uniquely solvable, and the operator $(I + \tilde R(x))$ has a bounded inverse. Now we are going to prove, that
all these facts follow from some simple properties of spectral data.

We will say that data $\left( \{ V(\la) \}_{\la > 0}, \{ \la_k, \al_k \}_{k = 1}^P \right)$ belongs to the class $\mbox{Sp}$, if

($i_1$) $\la_k$ are distinct negative numbers.

($i_2$) $\al_k$ are nonzero Hermitian matrices, $\al_k \ge 0$.

($i_3$) The $m \times m$ matrix function $\rho V(\la)$ is continuous and bounded as $\la > 0$, $V(\la) > 0$
and $M(\la) = O(\rho^{-1})$ as $\rho \to 0$, where $M(\la)$ is defined by \eqref{Msd}.

($i_4$) There exists a model problem $\tilde L$, such that \eqref{restV} holds.

Note that the spectral data of any self-adjoint boundary value problem $L(Q, h)$ belong to $\mbox{Sp}$.

\begin{lem}  \label{lem:mainsolve}
Let data $\left( \{ V(\la) \}_{\la > 0}, \{ \la_k, \al_k \}_{k = 1}^P \right)$ belong to $\mbox{Sp}$.
Then for each fixed $x \ge 0$, the system \eqref{maincont}-\eqref{maindisc} is uniquely solvable.
In other words, the operator $(I + \tilde R(x))$ is invertible.
\end{lem}

\begin{proof}
Fix $x \ge 0$. The operator $\tilde R(x)$ is compact, so it is sufficient to prove the unique solvability of 
the homogeneous system \eqref{maincont}-\eqref{maindisc}. Let
$\left(\{ \beta(x, \la) \}_{\la > 0}, \{ \beta_{ni}(x) \}_{(n, i) \in \mathcal{M}}\right) \in B_S$ be a solution
of the homogeneous system:
\begin{multline} \label{beta1}
   \beta(x, \la) + \int_0^{\iy} \beta(x, \mu) \hat V(\mu) \tilde D(x, \la, \mu) \, d\mu \\ +
   \sum_{k = 1}^P \beta_{k0}(x) \al_{k0} \tilde D(x, \la, \la_{k0}) - 
   \sum_{k = 1}^{\tilde P} \beta_{k1}(x) \al_{k1} \tilde D(x, \la, \la_{k1}) = 0_m, \quad \la > 0,
\end{multline}
\begin{multline*}
 	\beta_{ni}(x) + \int_0^{\iy} \beta(x, \mu) \hat V(\mu) \tilde D(x, \la_{ni}, \mu) \,d\mu \\ +
 	\sum_{k = 1}^P \beta_{k0}(x) \al_{k0} \tilde D(x, \la_{ni}, \la_{k0}) - 
 	\sum_{k = 1}^{\tilde P} \beta_{k1}(x) \al_{k1} \tilde D(x, \la_{ni}, \la_{k1}) = 0_m,
 	\quad (n, i) \in \mathcal{M}.
\end{multline*}
Note that the formula \eqref{beta1} gives an analytic continuation of the matrix function $\beta(x, \la)$
to the whole $\la$-plane. Clearly, $\beta(x, \la_{ni}) = \beta_{ni}(x)$.

Using the standard estimate $\| \tilde D(x, \la, \la_{kj}) \| \le \frac{C}{|\rho|} \exp(|\tau|x)$
(see \cite{FY07, FY01}) together with \eqref{restV}, one can show that 
\begin{equation} \label{estbeta}
 	\| \beta(x, \la) \| \le \frac{C}{|\rho|}\exp(|\tau|x).
\end{equation}

Define the function
\begin{multline} \label{defGamma}
\Gamma(x, \la) = -\int_0^{\iy} \beta(x, \mu) \hat V(\mu) \frac{\langle \tilde \vv^*(x, \mu), \tilde \Phi(x, \la) \rangle}{\la - \mu} d\mu \\
- \sum_{k = 1}^P \beta_{k0}(x) \al_{k0} \frac{\langle \tilde \vv^*_{k0}(x), \tilde \Phi(x, \la) \rangle}{\la - \la_{k0}} +
\sum_{k = 1}^{\tilde P} \beta_{k1}(x) \al_{k1} \frac{\langle \tilde \vv^*_{k1}(x), \tilde \Phi(x, \la) \rangle}{\la - \la_{k1}}.
\end{multline}
Using the relations $\tilde \Phi(x, \la) = \tilde S(x, \la) + \tilde \vv(x, \la) \tilde M(\la)$ and \eqref{beta1}, one can easily 
derive the following formula
\begin{multline*}
\Gamma(x, \la) = \beta(x, \la) \tilde M(\la) - \int_0^{\iy} \beta(x, \mu) \hat V(\mu) 
\frac{\langle \tilde \vv^*(x, \mu), \tilde S(x, \la) \rangle}{\la - \mu} d\mu \\ - 
\sum_{k = 1}^P \beta_{k0}(x) \al_{k0} \frac{\langle \tilde \vv^*_{k0}(x), \tilde S(x, \la) \rangle}{\la - \la_{k0}} + 
\sum_{k = 1}^{\tilde P} \beta_{k1}(x) \al_{k1} \frac{\langle \tilde \vv^*_{k1}(x), \tilde S(x, \la) \rangle}{\la - \la_{k1}}.
\end{multline*}
Since $\langle \tilde \vv^*(x, \mu), \tilde S(x, \la) \rangle_{x = 0} = -I_m$, we have
$$
 	\frac{\tilde \vv^*(x, \mu), \tilde S(x, \la)}{\la - \mu} = -\frac{I_m}{\la - \mu} + \int_0^x \tilde \vv^*(t, \mu) \tilde S(t, \la) \, dt.
$$
Consequently, we can represent $\Gamma(x, \la)$ in the following form
$$
 	\Gamma(x, \la) = \beta(x, \la) \tilde M(\la) + \int_0^{\iy} \frac{\beta(x, \mu) \hat V(\mu)}{\la - \mu} d\mu + 
 	\sum_{k = 1}^P \frac{\beta_{k0}(x) \al_{k0}}{\la - \la_{k0}} - \sum_{k = 1}^{\tilde P} \frac{\beta_{k1}(x) \al_{k1}}{\la - \la_{k1}}
 	+ \Gamma_1(x, \la),
$$
where the matrix function $\Gamma_1(x, \la)$ is entire by $\la$, since $\tilde S(x, \la)$ is entire.
Taking \eqref{Msd} into account, we obtain
$$
 	\Gamma(x, \la) = \beta(x, \la) M(\la) + \Gamma_1(x, \la) - \Gamma_2(x, \la), 
$$
where
$$
 	\Gamma_2(x, \la) = \int_0^{\iy} \frac{(\beta(x, \la) - \beta(x, \mu)) \hat V(\mu)}{\la - \mu}d\mu + 
 	\sum_{k = 1}^P \frac{(\beta(x, \la) - \beta_{k0}(x)) \al_{k0}}{\la - \la_{k0}} - 
 	\sum_{k = 1}^{\tilde P} \frac{(\beta(x, \la) - \beta_{k1}(x)) \al_{k1}}{\la - \la_{k1}}.
$$
Obviously, the function $\Gamma_2(x, \la)$ is entire in $\la$. Therefore, the function $\Gamma(x, \la)$
has simple poles at the points $\Lambda'$ and
\begin{equation*} 
 \Res_{\la = \la_{k0}} \Gamma(x, \la) = \beta_{k0}(x) \al_{k0}, \quad k = \overline{1, P}.
\end{equation*}
Furthermore, 
$$
 	\frac{1}{2 \pi i} \left( \Gamma^-(\la) - \Gamma^+(\la) \right) = \beta(x, \la) V(\la), \quad
 	\Gamma^{\pm}(\la) := \lim_{z \to 0, \, \mbox{Re}\, z > 0} \Gamma(\la \pm i z), \quad \la > 0.
$$
Using \eqref{estbeta} and the standard asymptotics for $\tilde \vv^*(x, \mu)$ and $\tilde \Phi(x, \la)$, one 
arrive at the estimate
\begin{equation} \label{estGamma}
\| \Gamma(x, \la) \| \le C |\rho|^{-2} \exp(-|\tau|x), \quad |\la| \to \iy.
\end{equation}

\begin{center}

\setlength{\unitlength}{1mm}
\begin{picture}(60, 50)
\linethickness{0.3mm}
\put(25,0){\vector(0,1){50}}
\put(0,25){\vector(1,0){60}}
\linethickness{0.15mm}
\put(45,27){\line(-1,0){4}}
\put(36,27){\vector(1,0){5}}
\put(36,23){\line(1,0){4}}
\put(45,23){\vector(-1,0){5}}
\put(20,34.5){\vector(2,1){0.5}}
\put(10,42){\vector(-1,-1){0.5}}
\put(2,48){$\ga_R^0$}
\linethickness{0.075mm}
\qbezier(25,36)(15,35)(14,25)
\qbezier(14,25)(15,15)(25,14)
\qbezier(25,36)(35,35)(36,27)
\qbezier(25,14)(35,15)(36,23)
\qbezier(45,27)(45,45)(25,47)
\qbezier(25,47)(3,45)(3,25)
\qbezier(3,25)(3,5)(25,3)
\qbezier(25,3)(45,5)(45,23)

\end{picture}

\medskip

Fig. 2
\end{center}

Introduce the matrix function $\mathscr{B}(x, \la) := \Gamma(x, \la) \beta^{\dagger}(x, \bar \la)$, and consider the integral
$$
 	\mathscr{I} := \frac{1}{2\pi i} \int_{\ga_R^0} \mathscr{B}(x, \la) \, d\la
$$
over the contour $\ga_R^0 := \left( \ga \cap \{ \la \colon \la \le R\} \right) \cup \{ \la \colon |\la| = R \}$
 (see Fig.~2). For a sufficiently large radius $R$, $\mathscr{I} = 0_m$ by
Cauchy theorem. In view of the estimates \eqref{estbeta}, \eqref{estGamma}, we have 
$\| \mathscr{B}(x, \la) \| \le C |\rho|^{-3}$, $|\la| \to \iy$. Hence
$$
 	\lim_{R \to \iy} \frac{1}{2 \pi i} \int_{|\la| = R} \mathscr{B}(x, \la) \, d\la = 0_m,
$$
$$
 	\frac{1}{2 \pi i} \int_{\ga} \mathscr{B}(x, \la) \, d\la = 0_m.
$$
The last integral over $\ga$ can be calculated by the residue theorem. It equals
$$
 	\sum_{k = 0}^{\iy} \beta_{k0}(x) \al_{k0} \beta_{k0}^{\dagger}(x)  + \frac{1}{2\pi i} 
 	\int_{|\la| = \eps} \mathscr{B}(x, \la) \, d\la + \int_{\eps}^{\iy} \beta(x, \la) V(\la) \beta^{\dagger}(x, \la) \, d\la,
$$
where $\eps > 0$ is sufficiently small. Note that since $M(\la) = O(\rho^{-1})$ as $\la \to 0$, 
the second term in the sum tends to zero as $\eps \to 0$. Finally, we obtain
$$
 	\sum_{k = 0}^{\iy} \beta_{k0}(x) \al_{k0} \beta_{k0}^{\dagger}(x) + \int_0^{\iy} \beta(x, \la) V(\la) \beta^{\dagger}(x, \la) \, d\la = 0.
$$
Since $\al_{k0} \ge 0$, $V(\la) > 0$, we get $\beta_{k0}(x) \al_{k0} = 0_m$, $\beta(x, \la) V(\la) = 0_m$, and
$\beta(x, \la) = 0$ for $\la > 0$. Since $\beta(x, \la)$ is an entire function in $\la$, we conclude that $\beta(x, \la) \equiv 0_m$.
Consequently, $\beta_{k0}(x) = \beta(x, \la_{k0}) = 0_m$. Thus, the homogeneous system has only trivial solution, 
so the system \eqref{maincont}-\eqref{maindisc} is uniquely solvable. 
\end{proof}

Solving the main equation, one can construct the following matrix functions:
\begin{multline} \label{defeps2}
 	\eps_0(x) := \int_0^{\iy} \vv(x, \mu) \hat V(\mu) \tilde \vv(x, \mu) \, d\mu +
 	\sum_{k = 1}^P \vv_{k0}(x) \al_{k0} \tilde \vv^*_{k0}(x) - 
 	\sum_{k = 1}^{\tilde P} \vv_{k1}(x) \al_{k1} \tilde \vv^*_{k1}(x),
 	\\ \eps(x) := -2 \eps_0'(x).
\end{multline}
and then recover $Q(x)$ and $h$ via \eqref{Qh}.
Theorem~\ref{thm:NSC} and Lemma~\ref{lem:mainsolve} yield the following theorem.

\begin{thm} \label{thm:SD}
For data $S := \left( \{ V(\la)\}_{\la > 0}, \: \{ \la_k, \al_k \}_{k = 1}^P \right)$
to be the spectral data of some self-adjoint boundary value problem $L(Q, h)$, $Q = Q^{\dagger}$, $h = h^{\dagger}$,
satisfying \eqref{xQx}, it is necessary and sufficient to belong to the class $\mbox{Sp}$ and
to have such a property, that $(1 + x) \eps(x) \in L((0, \iy); \mathbb{C}^{m \times m})$, where
$\eps(x)$ is constructed via \eqref{defeps2} by the unique solution of the system \eqref{maincont}-\eqref{maindisc} $\vv(x, \la)$. 
\end{thm}

\bigskip

{\large \bf 6. Perturbation of the discrete spectrum}\\

Return to the general non-self-adjoint problem, and consider one more particular case, 
when the solvability of the main equation \eqref{main} can be easily checked.
Let the problem $\tilde L$ be given, and $\tilde M(\la)$ is its Weyl matrix.
Consider the matrix function
\begin{equation} \label{discper}
M(\la) = \tilde M(\la) + \sum_{k = 1}^P \sum_{\nu = 1}^{m_k} \frac{\al_{k\nu}}{(\la - \la_k)^{\nu}},
\end{equation}
where $\la_k \in \mathbb{C}$ are some distinct numbers and $\al_{k\nu} \in \mathbb{C}^{m \times m}$, $k = \overline{1, P}$,
$\nu = \overline{1, m_k}$. Then $\hat V(\la) = 0_m$, and by virtue of the residue theorem, 
the main equation \eqref{main} takes the form
$$
 	\tilde \vv(x, \la) = \vv(x, \la) + \sum_{k = 1}^P \sum_{i = 0}^{m_k - 1} \frac{\partial^i}{\partial \la^i} \vv(x, \la_k)
 	\sum_{\nu = i + 1}^{m_k} \al_{k\nu} \tilde D_{\langle 0, \nu - i - 1 \rangle}(x, \la, \la_k),
$$
where $\tilde D_{\langle i, j \rangle}(x, \la, \mu) := \dfrac{\partial^{i + j}}{\partial \la^i \partial \mu^j} \tilde D(x, \la, \mu)$.
Differentiating this relation with respect to $\la$, we arrive at the following system
of linear algebraic equations with respect to the unknown variables $\left\{ \dfrac{\partial^s}{\partial \la^s} \vv(x, \la_n) \right\}$:
\begin{equation} \label{main2}
\frac{\partial^s}{\partial \la^s} \tilde \vv(x, \la_n) = 
\frac{\partial^s}{\partial \la^s} \vv(x, \la_n) + \sum_{k = 1}^P \sum_{i = 0}^{m_k - 1} \frac{\partial^i}{\partial \la^i} \tilde \vv(x, \la_k)
\sum_{\nu = i + 1}^{m_k} \al_{k\nu} \tilde D_{\langle s, \nu - i - 1 \rangle}(x, \la_n, \la_k), 
\end{equation}
$n = \overline{1, P}$, $s = \overline{0, m_n - 1}$.
The system \eqref{main2} has a unique solution if and only if its determinant is not zero.
Having the solution of \eqref{main2}, one can construct
\begin{equation} \label{defeps3}
  	\eps_0(x) = \sum_{k =1}^P \sum_{i = 0}^{m_k - 1} \frac{\partial^i}{\partial \la^i} \vv(x, \la_k) 
  	\sum_{\nu = i + 1}^{m_k} \al_{k \nu} \frac{\partial^{\nu - i - 1}}{\partial \la^{\nu - i - 1}} \tilde \vv^*(x, \la_k),
  	\quad \eps(x) = -2 \eps_0'(x),
\end{equation}
and then find $Q(x)$ and $h$ via \eqref{Qh}.

\begin{thm}  \label{thm:finite}
For the matrix function $M(\la)$ in the form \eqref{discper} to be the Weyl matrix of a certain boundary value
problem $L$, it is necessary and sufficient, that the determinant of the system \eqref{main2} differs from zero,
and $\eps(x) \in L((0, \iy); \mathbb{C}^{m \times m})$, where $\eps(x)$ is defined in \eqref{defeps3}, 
\end{thm}

There is the example, provided in \cite[Section 2.3.2]{FY01}, showing that even in the simple case of a finite perturbation,
the condition $\eps(x) \in L((0, \iy); \mathbb{C}^{m \times m})$ is essential and can not be omitted.
So it is crucial in Theorems~\ref{thm:NSC}, \ref{thm:SD} and \ref{thm:finite}.

\begin{comment}
\bigskip

{\large \bf Competing interests}\\

The author declares that she has no competing interests.

\medskip

{\large \bf Author's contributions}\\

The author has solely obtained the results and prepared the manuscript.

\end{comment}

\medskip

{\bf Acknowledgment.} This work was supported by Grant 1.1436.2014K
of the Russian Ministry of Education and Science and by Grants 13-01-00134 and 14-01-31042
of Russian Foundation for Basic Research.

\vspace{1cm}

Natalia Bondarenko

Department of Mechanics and Mathematics

Saratov State University

Astrakhanskaya 83, Saratov 410012, Russia

{\it bondarenkonp@info.sgu.ru}

\end{document}